\documentclass[10pt]{article}
\usepackage{nips14submit_e,times}

\pdfoutput=1
\newcommand{\StatPaper}{true}
\newcommand{\LongPaper}{true}

\nipsfinalcopy

\usepackage[latin1]{inputenc}
\usepackage{mystyle}

\title{Model Consistency of \\ Partly Smooth Regularizers}

\author{
Samuel Vaiter and Gabriel Peyr\'e \\
CNRS and CEREMADE\\
Univ. Paris-Dauphine\\
{\small  \url{ {vaiter,peyre}@ceremade.dauphine.fr } } \\
\And
Jalal Fadili \\
GREYC \\ 
CNRS-ENSICAEN \\
{\small  \url{Jalal.Fadili@ensicaen.fr} }
}

\begin{document}

\maketitle

\begin{abstract}

This paper studies least-square regression penalized with partly smooth convex regularizers. This class of functions is very large and versatile allowing to promote solutions conforming to some notion of low-complexity. Indeed, they force solutions of variational problems to belong to a low-dim\-ensional manifold (the so-called model) which is stable under small perturbations of the function. This property is crucial to make the underlying low-complexity model robust to small noise. 
We show that a generalized ``irrepresentable condition'' implies stable model selection under small noise perturbations in the observations and the design matrix, when the regularization parameter is tuned proportionally to the noise level. This condition is shown to be almost a necessary condition.
We then show that this condition implies model consistency of the regularized estimator. That is, with a probability tending to one as the number of measurements increases, the regularized estimator belongs to the correct low-dimensional model manifold.
This work unifies and generalizes several previous ones, where model consistency is known to hold for sparse, group sparse, total variation and low-rank regularizations.
\OnlyLong{Lastly, we also show that this generalized ``irrepresentable condition'' implies that the forward-backward proximal splitting algorithm identifies the model after a finite number of steps. }


\end{abstract}


\section{Introduction}

\subsection{Problem Statement}

We consider the following observation model
\eq{
	y = \Phi \x_0 + w, 
}
where $\Phi \in \RR^{\P \times \N}$ is the design matrix (in statistics or machine learning) or the forward operator (in signal and imaging sciences), $\x_0 \in \RR^\N$ is the vector to recover and $w \in \RR^\P$ is the noise. The design can be either deterministic or random, and similarly for the noise $w$.

Regularization is now a central theme in many fields including statistics, machine learning and inverse problems. It allows one to impose on the set of candidate solutions some prior structure on the object $x_0$ to be estimated. We therefore consider a positive convex and finite-valued function $J$ to promote such a prior. This then leads to solving the following convex optimization problem
\eql{\label{eq-regularization-lagrangian-init}
	\umin{\x \in \RR^\N}
	\left\{ J(\x) + \frac{1}{2\la}\norm{\Phi \x - y }^2 \right\}, 
}
where $\la > 0$ controls the amount of regularization.

To simplify the notations, we introduce the following ``canonical'' parameters
\eq{
	\th = (\mu,u,\Corr) = \pa{ \frac{\la}{\P}, \frac{\Phi^* y}{\P}, \frac{\Phi^*\Phi}{\P}  }
	\in \Theta = \RR^+ \times \RR^\N \times \RR^{\N \times \N}
}
and we denote
\eq{
	\epsilon = \frac{\Phi^* w}{\P} = \u - \Corr \x_0.
}
In the following, we assume without loss of generality that $y \in \Im \Phi$ and thus $\u \in \Im(\Corr)$.

With these new parameters, the initial problem~\eqref{eq-regularization-lagrangian-init} now reads
\eql{\label{eq-regularization-lagrangian} \tag{$\Pp_\th$}
	\umin{\x \in \RR^\N}
	\left\{ 
		 \ener(\x,\th) = 
		  J(\x)  + \frac{1}{2\mu}\dotp{\Corr \x}{\x} - \frac{1}{\mu}\dotp{\x}{\u} + \frac{1}{2\mu}\dotp{\Corr^+ \u}{\u}
	\right\}.
}
where $A^+$ stands for the Moore-Penrose pseudo-inverse of a matrix $A$.

When $\mu \rightarrow 0^+$, we consider the constrained problem 
\eql{\label{eq-regularization-noiseless}\tag{$\Pp_{\th_0}$}
	\umin{\x \in \RR^\N}
	\left\{ \ener(\x, \th_0 ) = J(\x) + \iota_{\Hh_u}(\x) \right\}
	\qwhereq
	\Hh_u = \enscond{\x \in \RR^\N}{\Corr \x = u}
}
where $\th_0 = (0,\u,\Corr)$ and where the indicator function of some closed convex set $\Cc$ is $\iota_\Cc(\x)=0$ for $\x \in \Cc$ and $\iota_\Cc(\x)=+\infty$ otherwise. With these notations, $\ener$ is a function on $\RR^\N \times \Theta$. 


The goal of this paper is to asses the recovery performance of \eqref{eq-regularization-lagrangian}, i.e. to understand how close is the recovered solution of~\eqref{eq-regularization-lagrangian} to $\x_0$. We focus here on the low noise regime, i.e. when $\epsilon$ is small enough, and study not only $\ell^2$ stability, but also the identifiability of the correct low-dimensional manifold associated to $\x_0$. This unifies and extend a large body of literature, including sparsity and low-rank regularization, which turn to be a special case of the general theory of partly-smooth regularization.

\subsection{Notations}



If $\Mm \subset \RR^\N$ is a $\Cdeux$-manifold around $\x \in \RR^\N$, we denote $\tgtManif{\x}{\Mm}$ the tangent space of $\Mm$ at $\x \in \RR^\N$. We define the tangent model subspace as
\eq{
	T_{\x} = \VecHull(\partial J(\x))^\bot.
}
where the linear hull of a convex set $\Cc \subset \RR^\N$ is 
	$\VecHull(\Cc) = \enscond{ \rho (\x-\x')  }{ (\x,\x') \in \Cc^2, \rho \in \RR}$.
For a convex set $\Cc \subset \RR^\N$, $\ri(\Cc)$ is its relative interior, i.e. its interior for the topology of its affine hull (the smallest affine space containing $\Cc$). 
For a linear space $T$, we denote $P_T$ the orthogonal projection on $T$ and for a matrix $\Corr \in \RR^{\N \times \N}$, $\Corr_T = P_T \Corr P_T$.


\section{Partly-smooth Functions}

Toward the goal of studying the recovery guarantees of problem \eqref{eq-regularization-lagrangian}, our central assumption will be that $J$ is a partly smooth function. Partial smoothness of functions was originally defined~\cite{Lewis-PartlySmooth}. Our definition hereafter specializes it to the case of finite-valued convex functions.

\begin{defn} \label{dfn-partly-smooth}
	Let $J$ be a finite-valued convex function. 
	$J$ is \emph{partly smooth at $\x$ relative to a set $\Mm$} containing $\x$ if 
	\begin{enumerate}[(i)]\setlength{\itemsep}{0pt}
		\item\label{PS-C2}(Smoothness) $\Mm$ is a $\Cdeux$-manifold around $\x$ and $J$ restricted to $\Mm$ is $\Cdeux$ around $\x$.
                \item\label{PS-Sharp}(Sharpness) The tangent space $\tgtManif{\x}{\Mm}$ is $T_{\x}$.
		\item\label{PS-DiffCont}(Continuity) The set-valued mapping $\partial J$ is continuous at $\x$ relative to~$\Mm$.
	\end{enumerate}
        $J$ is said to be \emph{partly smooth relative to a set $\Mm$} if $\Mm$ is a manifold and $J$ is partly smooth at each point $\x \in \Mm$ relative to $\Mm$.
        $J$ is said to be \emph{locally partly smooth at $\x$ relative to a set $\Mm$} if $\Mm$ is a manifold and there exists a neighbourhood $U$ of $\x$ such that $J$ is partly smooth at each point $\x' \in \Mm \cap U$ relative to $\Mm$.
\end{defn}

Note that in the previous definition, $\Mm$ needs only to be defined locally around $\x$, and it can be shown to be locally unique thanks to prox-regularity of proper closed convex functions, see~\cite[Corollary~ 4.2]{HareLewis04}. 

\begin{rem}[Discussion of the properties]\label{rem:assump}
Since $J$ is convex continuous, the subdifferential of $\partial J(\x)$ is everywhere non-empty and compact and every subgradient is regular. Therefore, the Clarke regularity property~\cite[Definition~2.7(ii)]{Lewis-PartlySmooth} is automatically verified.
In view of~\cite[Proposition~2.4(i)-(iii)]{Lewis-PartlySmooth}, the sharpness property \eqref{PS-Sharp} is equivalent to ~\cite[Definition~2.7(iii)]{Lewis-PartlySmooth}.
The continuity property \eqref{PS-DiffCont} is equivalent to the fact that $\partial J$ is inner semicontinuous at $\x$ relative to $\Mm$, that is: for any sequence $\x_n$ in $\Mm$ converging to $\x$ and any $\eta \in \partial J(\x)$, there exists a sequence of subgradients $\eta_n \in \partial J(\x_n)$ converging to $\eta$. This equivalent characterization will be very useful in the proof of our main result.
\end{rem}

\subsection{Examples in Imaging and Machine Learning}
\label{sec-examples}

We describe below some popular examples of partly smooth regularizers that are routinely used in machine learning, statistics and imaging sciences. 

\paragraph{$\ell^1$ sparsity.}

One of the most popular non-quadratic convex regularization is the $\ell^1$ norm
$J(\x) = \sum_{i=1}^{\N} |\x_i|$,  
which promotes sparsity. Indeed, it is easy to check that $J$ is partly smooth at $\x$ relative to the subspace 
\eq{
	\Mm = T_{\x} = \enscond{ u \in \RR^\N }{ \supp(u) \subseteq \supp(\x) }. 
}
The use of sparse regularizations has been popularized in the signal processing literature under the name basis pursuit method~\cite{chen1999atomi} and in the statistics literature under the name Lasso~\cite{tibshirani1996regre}. 

\paragraph{$\ell^1-\ell^2$ group sparsity.}

To better capture the sparsity pattern of natural signals and images, it is useful to structure the sparsity into non-overlapping blocks/groups $\Bb$ such that $\bigcup_{b \in \Bb} b = \{1,\ldots,\N\}$. This group structure is enforced by using typically the mixed $\ell^1-\ell^2$ norm
$J(\x) = \sum_{b \in \Bb} \norm{\x_b}$, 
where $\x_b = (\x_i)_{i \in b} \in \RR^{|b|}$. We refer to~\cite{yuan2005model,bach2008consistency} and references therein for more details. Unlike the $\ell^1$ norm, and except the case $|b|=1$, the $\ell^1-\ell^2$ norm is not polyhedral, but is still partly smooth at $\x$ relative to the linear manifold defined as
\eq{
	\Mm = T_{\x} = \enscond{ \x' }{ \supp_\Bb(\x') \subseteq \supp_\Bb(\x) }
	\qwhereq
	\supp_\Bb(\x) = \bigcup \enscond{b}{ \x_{b} \neq 0 }.
}

\paragraph{Spectral functions.}

The natural spectral extension of sparsity to matrix-valued data $\x \in \RR^{\N_0 \times \N_0}$ (where $\N=\N_0^2$) is to impose a low-rank prior, which should be understood as sparsity of the singular values. Denote $\x = V_{\x} \diag(\La_{\x}) U_{\x}^*$ an SVD decomposition of $\x$, where $\La_{\x} \in \RR_+^{\N_0}$. Note that this can be extended easily to rectangular matrices. The nuclear norm is defined as $J(\x) = \norm{\x}_* = \norm{\La_{\x}}_1$. 
It has been used for instance in machine learning applications~\cite{bach2008consistency}, matrix completion~\cite{recht2010guaranteed,candesExactCompletion} and phase retrieval~\cite{CandesPhaseLift}. The nuclear norm can be shown to be partly smooth at $\xx$ relative to the manifold~\cite[Example~2]{LewisMalick08}
$\Mm = \enscond{\x'}{ \rank(\x')=\rank(\x) }$.
More generally, if $j : \RR^{\N_0} \rightarrow \RR$ is a permutation-invariant closed convex function, then one can consider the function 
$J(\x) = j(\La_{\x})$
which can be shown to be a convex function as well~\cite{LewisMathEig}. When restricted to the linear space of symmetric matrices, $j$ is partly smooth at $\La_{\x}$ for a manifold $m_{\La_{\x}}$, if and only if $J$ is partly smooth at $\x$ relative to the manifold
\eq{
	\Mm = \enscond{ U \diag(\La) U^* }{\La \in m_{\La_{\x}}, U \in \Oo_{\N_0}}, 
}
where $\Oo_{\N_0} \subset \RR^{\N_0 \times \N_0}$ is the group of orthogonal matrices. 
This result is proved in~\cite[Theorem~3.19]{Daniilidis-SpectralIdent}, extending the initial work of~\cite{DANIILIDIS-SpectralManif}. This result can be extended to non-symmetric matrices by requiring that $j$ is an absolutely permutation-invariant closed convex function, see~\cite[Theorem~5.3]{Daniilidis-SpectralIdent}. The nuclear norm $\norm{\cdot}_*$ is a special case where $j(\La)=\norm{\La}_1$.

\paragraph{Analysis regularizers.}

If $J_0 : \RR^q \rightarrow \RR$ is a convex function and $D \in \RR^{\N \times q}$ is a linear operator, one can consider the analysis regularizer 
$J(\x) = J_0(D^* \x)$. 
A popular example is when taking $J_0=\norm{\cdot}_1$ and $D^*=\nabla$ a finite difference approximation of the gradient of an image. This defines the (anisotropic) total variation, which promotes piecewise constant images, and is popular in image processing~\cite{rudin1992nonlinear}. It is also possible to define families of sparsity-enforcing prior by using $J_0=\norm{\cdot}_*$ the nuclear norm, see~\cite{Grave-TraceLasso,RichardBV13}.  If $J_0$ is partly smooth at $z=D^* \x$ for the manifold $\Mm_z^0$, then it is shown in~\cite[Theorem 4.2]{Lewis-PartlySmooth} that $J$ is partly smooth at $\x$ relative to the manifold
\eq{
	\Mm = \enscond{ \x' \in \RR^\N }{ D^* \x' \in \Mm_z^0 }.
}
\OnlyLong{Note that as $J_0$ is convex and continuous, so is $J$, and there is no need of the transversality/regularity condition in~\cite[Theorem 4.2]{Lewis-PartlySmooth}.}

\paragraph{Mixed regularization.}

Starting from a set of convex functions $\{J_\ell\}_{\ell \in \Ll}$, it is possible to design a convex function as 
$J_\ell(\x) = \sum_{\ell \in \Ll} \rho_\ell J_{\ell}(\x)$, 
where $\rho_\ell > 0$ are weights. A popular example is to impose both sparsity and low rank of a matrix, when using $J_1=\norm{\cdot}_1$ and $J_2=\norm{\cdot}_*$, see for instance~\cite{oymak2012simultaneously}. 
If each $J_\ell$ is partly smooth at $\x$ relative to a manifold $\Mm^\ell$, then it is shown in~\cite[Corollary 4.8]{Lewis-PartlySmooth} that $J$ is also partly smooth at $\x$ for $\Mm = \bigcap_{\ell \in \Ll} \Mm^\ell$.
\OnlyLong{Again, the regularity condition in~\cite[Corollary 4.8]{Lewis-PartlySmooth} is in force in our case by convexity and continuity.}


\section{Main results}

In the following, we denote $T=T_{\x_0}$, $e = P_{T}( \partial J(\x_0) ) \in \RR^{\N}$. 
Before stating our main contributions, we first introduce a central object of this paper, which controls the stability of $\Mm$ when the signal to noise ratio is large enough. 

\begin{defn}[Linearized pre-certificate] \label{defn-etaF}
For some matrix $\Corr \in \RR^{\N \times \N}$, assuming $\ker(\Corr) \cap T = \{0\}$, we define
$\etaL{\Corr} = \Corr \Corr_T^+ e$. 
\end{defn}

\subsection{Deterministic model consistency.}

We first consider the case where $\Phi$ and $w$ (or equivalently $\Corr$ and $u$) are fixed and deterministic. Our main contribution is the following theorem, which shows the robustness of the manifold $\Mm$ associated to $\x_0$ to small perturbations on both the observations and the design matrix, provided that $\mu$ is well chosen.

\begin{thm}\label{thm-stability}
	We assume that $J$ is locally partly smooth at $\x_0$ relative to $\Mm$ and that there exists $\CorrL \in \RR^{\N \times \N}$ such that 
	\eql{\label{eq-hyp-etaF}
		\ker(\CorrL) \cap T = \{0\}, 
		\qandq
		\etaL{\CorrL} \in \ri(\partial J(\x_0)).
	}
	Then, there exists a constant $C>0$ such that if 
	\eql{\label{eq-condition-closeness}
		\max\pa{
			\norm{\Corr - \CorrL}, 
	      	\norm{\epsilon} \mu^{-1}, 
	      	\mu
	     } \leq C,
	}
	the solution $\x_{\th}$ of~\eqref{eq-regularization-lagrangian} is unique and satisfies 
	\eql{\label{eq-model-consistency}
		\x_\th \in \Mm \qandq \norm{\x_\th-\x_0} = O(\norm{\epsilon}) .
	} 
\end{thm}

This theorem is proved in Section~\ref{sec-proof-thm-stability}.

\OnlyLong{
\begin{rem}[Deterministic vs. randomized settings]
A typical case of application of this result is in inverse problem for imaging. In this setting, $\Phi$ is fixed, so that one directly uses $\CorrL = \Corr = \Phi^*\Phi/\P$. In contrast (as detailed in Theorem~\ref{thm-consistency}), in statistics or machine learning, one rather considers the asymptotic regime where the number of rows of $\Phi$ increases, so that $\Corr$ only reach $\CorrL$ in the limit $\P \rightarrow +\infty$. 
\end{rem}

\begin{rem}[Identification of the manifold]
	Theorem~\ref{thm-stability} guarantees that, under some hypotheses on $\x_0$ and $\th$, $\x_\th$ belongs to $\Mm$. For all the regularizations considered in Section~\ref{sec-examples}, one can furthermore show that actually, under these hypotheses, $\Mm_{\x_\th} = \Mm$. This is because, for any $(\x,\x')$ with $\x' \in \Mm_{\x}$ close enough to $\x$, one has $\Mm_{\x'}=\Mm_{\x}$.
\end{rem}
}{}

The following proposition, proved in Section~\ref{sec-proof-prop-instability}, shows that Theorem~\ref{thm-stability} is in some sense sharp, since the hypothesis $\etaL{\Corr} \in \ri(\partial J(\x_0))$ (almost) characterizes the stability of $\Mm$. 

\begin{prop}\label{prop-instability}
	We suppose that $\x_0$ is the unique solution of $\Pp_{(0,\CorrL \x_0, \CorrL)}$ and that
	\eql{\label{eq-instability-cond}
		\ker(\CorrL) \cap T = \{0\}, 
		\qandq
		\etaL{\CorrL} \notin \partial J(\x_0).
	}
	Then there exists $C>0$ such that if~\eqref{eq-condition-closeness} holds, then  any solution $\x_\th$ of~\eqref{eq-regularization-lagrangian} for $\mu>0$ satisfies $\x_\th \notin \Mm$.  
\end{prop}

In the particular case where $\epsilon=0$ (no noise) and $\CorrL=\Corr$, this result shows that the manifold $\Mm$ is not correctly identified when solving $\Pp_{(\mu,\Corr \x_0,\Corr)}$ for any $\mu>0$ small enough. 

\begin{rem}[Critical case]
The only case not covered by either Theorem~\ref{thm-stability} or Proposition~\ref{prop-instability} is when $\etaL{\CorrL} \in \text{rbound}(\partial J(\x_0))$ (the relative boundary). In this case, one cannot conclude, since depending on the noise $w$, one can have either stability or non-stability of $\Mm$. We refer to~\cite{vaiter-analysis} where an example illustrates this situation for the 1-D total variation $J=\norm{\nabla \cdot}_1$ (here $\nabla$ is a discretization of the 1-D derivative operator).
\end{rem}

\subsection{Probabilistic model consistency.}

We now turn to study consistency of our estimator. In this section, we work under the classical setting where $\N$ and $\x_0$ are fixed as the number of observations $\P \to \infty$. We consider that the design matrix and the noise are random. More precisely, the data $(\Phirow_i,w_i)$ are random vectors in $\RR^\N \times \RR$, $i=1,\cdots,n$, where $\Phirow_i$ is the $i$-th row of $\Phi$, are assumed independent and identically distributed (i.i.d.) samples from a joint probability distribution such that $\EE\pa{w_i|\Phirow_i}=0$, finite fourth-order moments, i.e. $\EE\pa{w_i^4} < +\infty$ and $\EE\pa{\norm{\Phirow_i}^4} < +\infty$. Note that in general, $w_i$ and $\Phirow_i$ are not necessarily independent. 
It is possible to extend our result to other distribution models by weakening some of the assumptions and strenghthening others, see e.g. \cite{KnightFu2000,Zhao-irrepresentability,bach2008consistency}. Let's denote $\CorrL = \EE(\xi^* \xi) \in \RR^{\N \times \N}$, where $\xi$ is any row of $\Phi$. We do not make any assumption on invertibility of~$\CorrL$. 

To make the discussion clearer, the canonical parameters $\th$ will be indexed by $\P$. The estimator $\x_{\th_\P}$ obtained by solving $(\Pp_{\th_\P})$ for a sequence $\th_\P$ is said to be consistent for $\x_0$ if, $\lim_{\P \to +\infty} \Pr\pa{\x_{\th_\P} ~ \text{is unique}} \to 1$ and $\x_{\th_\P}$ converges to $\x_0$ in probability. The estimator is said to be model consistent if $\lim_{\P \to +\infty} \Pr\pa{\x_{\th_\P} \in \Mm} \to 1$, where $\Mm$ is the manifold associated to $\x_0$.

The following result ensures model consistency for certain scaling of $\mu_\P$. It is proved in Section~\ref{sec-proof-thm-consistency}

\begin{thm}\label{thm-consistency}
	If conditions~\eqref{eq-hyp-etaF} hold and 
	\eql{\label{eq-lambda-scaling}
		\mu_\P = o(1)
		\qandq
		\mu_\P^{-1} = o(\P^{1/2}).
	} 
	Then the estimator $\x_{\th_\P}$ of $\x_0$ obtained by solving $(\Pp_{\th_\P})$ is model consistent. 
\end{thm}

\OnlyLong{
\begin{rem}[Sharpness of the criterion]
	Conversely, if $\x_0$ is the unique solution of $\Pp_{0,\Corr \x_0,\Corr}$, conditions~\eqref{eq-instability-cond} and~\eqref{eq-lambda-scaling} hold, one shows that the estimator of $\x_0$ defined by~\eqref{eq-regularization-lagrangian} is not model consistent. 
\end{rem}
}

\OnlyLong{
\subsection{Algorithmic Implications}

A popular scheme to compute a solution of~\eqref{eq-regularization-lagrangian} is the Forward-Backward splitting algorithm. A comprehensive treatment of the convergence properties of this algorithm, and other proximal splitting schemes, can be found in the monograph~\cite{Bauschke-book}. Starting from some $\x_0 \in \RR^\N$, the algorithm implements the following iteration
\eq{
	\x_{k+1} = \text{Prox}_{\tau\mu J}\pa{ \x_k + \tau ( u - \Corr \x_k ) },  
}
where the step size satisfies $0 < \tau < 2/\norm{\Corr}$, and the proximity operator is defined as, for $\ga > 0$
\eq{
	\text{Prox}_{\ga J}(\x) = \uargmin{\x' \in \RR^\N} \frac{1}{2}\norm{\x-\x'}^2 + \ga J(\x').
}

The following theorem shows that the Forward-Backward algorithm correctly identifies the manifold $\Mm$ after a finite number of iterations. 

\begin{thm}\label{thm-fb}
	Suppose that the assumptions of Theorem~\ref{thm-stability} hold. Then, for $k$ large enough, $\x_k \in \Mm$.
\end{thm}
\begin{proof} A close inspection of the proof of Theorem~\ref{thm-stability} shows that the solution $\x_\th$ of~\eqref{eq-regularization-lagrangian}  is unique and that the vector $\eta_\th = \frac{u- \Corr \x_\th}{\la}$ satisfies $\eta_\th \in \ri(\partial J(\x_\th))$ when \eqref{eq-hyp-etaF} and \eqref{eq-condition-closeness} hold. This in turn implies that the assumptions of~\cite[Theorem~13.7]{HareFB11}, see also~\cite[Theorem~2]{Hare-Lewis-Algo}\footnote{The result of~\cite{HareFB11} applies more generally to variable metric (Newton-like) Forward-Backward when the smooth term is assumed to be $\Cdeux$. This can be easily adapted to our case by taking the metric as the identity. Observe also that the result of~\cite{Hare-Lewis-Algo} applies to the projected gradient algorithm, i.e. when $J$ is the indicator function of a closed convex partly smooth set, and the proof easily extends also to an arbitrary partly smooth closed convex function.}, are fulfilled and thus shows the announced manifold identification result. 
\end{proof}

This result sheds some light on the convergence behaviour of this algorithm in the favourable case where condition~\eqref{eq-hyp-etaF} holds and $(\norm{\Corr-\CorrL}, \norm{\epsilon}/\mu,\mu)$ are sufficiently small. 
}

\subsection{Relation to Previous Works}

\OnlyLong{
\paragraph{Works on linear convergence rates.}

Following the pioneer work~\cite{burger2004convergence} (who study convergence in term of Bregman divergence), there is a large amount of works on the study conditions under which $\norm{\x_{\th}-\x_0}=O(\norm{\epsilon})$ (so-called linear convergence rate) where $\x_{\th}$ is any solution of~\eqref{eq-regularization-lagrangian}, see for instance the book~\cite{ScherzerBook09} for an overview of these results.  The initial work of~\cite{GrassmaiCPAM} proves a sharp criteria to ensure linear convergence rate for the $\ell^1$ norm, and this approach is further extended to arbitrary convex functions by~\cite{GrasmairPositively11} and~\cite{2013-sampta-decomposable} who proves respectively convergence rate in term of $J$ functional and $\ell^2$ norm. 

These works show that if 
\eql{\label{eq-source-condition}
	\ker(\Corr) \cap T = \{0\}
	\qandq
	\exists \eta \in \Im(\Corr) \cap \ri(\partial J(\x_0))
}
(which is often called the source condition), then linear convergence rate holds. Note that condition~\eqref{eq-hyp-etaF} implies~\eqref{eq-source-condition}, but it is stronger. Indeed, condition~\eqref{eq-source-condition} does not ensure model consistency~\eqref{eq-model-consistency}, which is a stronger requirement. Model consistency requires, as we show in our work, the use of a special certificate, the minimal norm certificate $\eta_0$, which is equal to $\etaL{\Corr}$ if $\etaL{\Corr} \in \ri(\partial J(\x_0))$ (see Proposition~\ref{prop:uniqueness}). 
}

\OnlyLong{\paragraph{Works on model consistency.}}

Theorem~\ref{thm-stability} is a generalization of a large body of results in the literature. For the Lasso, i.e. $J=\norm{\cdot}_1$, and when $\Corr=\CorrL$, to the best of our knowledge, this result was initially stated in~\cite{fuchs2004on-sp}. In this setting, the result~\eqref{eq-model-consistency} corresponds to the correct identification of the support, i.e. $\supp(\x_\th)=\supp(\x_0)$. 
Condition~\eqref{eq-hyp-etaF} for $J=\norm{\cdot}_1$ is known in the statistics literature under the name ``irrepresentable condition'', see e.g. \cite{Zhao-irrepresentability}. 
\cite{KnightFu2000} have shown estimation consistency for Lasso for fixed $\N$ and $\x_0$ and asymptotic normality of the estimates.
The authors in \cite{Zhao-irrepresentability} proved Theorem~\ref{thm-consistency} for $J=\norm{\cdot}_1$, though under slightly different assumptions on the covariance and noise distribution.
A similar result was established in \cite{Jia-ElasticNet-Consistency} for the elastic net, i.e. $J=\norm{\cdot}_1 + \rho \norm{\cdot}_2^2$ for $\rho > 0$.
In \cite{Bach08group} and~\cite{bach2008consistency}, the author has shown Theorem~\ref{thm-consistency} for two special cases, namely the group Lasso nuclear/trace norm minimization. Note that these previous works assume that the asymptotic covariance $\CorrL$ is invertible. We do not impose such an assumption, and only require the weaker restricted injectivity condition $\ker(\CorrL) \cap T = \{0\}$. 
In a previous work~\cite{vaiter-analysis}, we have proved an instance of Theorem~\ref{thm-stability} when $\Corr=\CorrL$ and $J(\x) = \norm{D^* \x}_1$, where $D \in \RR^{\N \times q}$ is an arbitrary linear operator. This covers as special cases the discrete anisotropic total variation or the fused Lasso. 
This result was further generalized in~\cite{vaiter2013model} when $\Corr=\CorrL$, and $J$ belongs to the class of partly smooth functions relative to  linear manifolds $\Mm$, i.e. $\Mm=T_{\x}$. Typical instances encompassed in this class are the $\ell^1-\ell^2$ norm, or its analysis version, as well as polyhedral gauges including the $\ell^\infty$ norm. Note that the nuclear norm (and composition of it with linear operators as proposed for instance in~\cite{Grave-TraceLasso,RichardBV13}), whose manifold is not linear, does not fit into the framework of~\cite{vaiter2013model}, while it is covered by Theorem~\ref{thm-stability}. 
\OnlyLong{Lastly, a similar result was proved in~\cite{2013-duval-sparsespikes} for an infinite dimensional sparse recovery problem over space of measures, when $J$ the total variation of a measure. In this setting, a interesting finding is that, when $\eta_0 \in \ri(\partial J(\x_0))$, $\eta_0$ is not equal to $\etaL{\Phi^*\Phi}$ but to a difference certificate (called ``vanishing derivative'' certificate in~\cite{2013-duval-sparsespikes}) that can also be computed by solving a linear system.}
\OnlyLong{
Condition~\eqref{eq-hyp-etaF} is often used when $\Phi$ is drawn from the Gaussian matrix ensemble to asses the performance of compressed sensing recovery with $\ell^1$ norm, see~\cite{wainwright-sharp-thresh,dossal2011noisy}. This is extended to a more general family of decomposable norms (including in particular $\ell^1-\ell^2$ norms and the nuclear norm) in~\cite{candes2011simple}, but only in the noiseless setting. Our result shows that this analysis extends to the noisy setting as well, and ensures model consistency at high signal to low noise levels. 
The same condition is used to asses the performance of matrix completion (i.e. the operator $\Phi$ is a random masking operator) in a noiseless setting~\cite{candesExactCompletion,candesT09}.
It was also used to ensure $\ell^2$ robustness of matrix completion in a noisy setting~\cite{CandesP10}, and our findings shows that these results also ensure rank consistency for matrix completion at high signal to low noise levels.}

\OnlyLong{
\paragraph{Sensitivity analysis.} 

Theorem~\ref{thm-stability} can be seen as a sensitivity analysis of the minimizers of the function $f$ at the point $(\x,\th)=(\x_0,\th_0^0)$. Classical sensitivity analysis of non-smooth functions seeks condition to ensure continuity of the map $\th \mapsto \x_\th$, see for instance~\cite{rockafellar1998variational}. This is usually guaranteed by the source condition~\eqref{eq-source-condition}, which, as already exposed, ensures linear convergence rate, and hence Lipschitz behaviour of this map. The analysis proposed by Theorem~\ref{thm-stability} goes one step further, by assessing that $\Mm_{\x_0}$ is a stable manifold (in the sense of~\cite{Wright-IdentSurf}), since the minimizer $\x_\th$ is unique and stays in $\Mm_{\x_0}$ for small $\th$. Our main source of inspiration for this analysis is the notion of partly smooth function introduced by Lewis~\cite{Lewis-PartlySmooth} in order to ensure the existence of stable manifolds. 
For convex functions (which is the setting considered in our work) this corresponds to the notion of $\Uu$-Lagrangian, introduced in~\citep{Lemarechal-ULagrangian}.
Loosely speaking, a partly smooth function behaves smoothly as we move on the identifiable manifold, and sharply if we move normal to the manifold. In fact, the behaviour of the function and of its minimizers (or critical points) depend essentially on its restriction to this manifold, hence offering a powerful framework for sensitivity analysis theory. In particular, critical points of partly smooth functions move stably on the manifold as the function undergoes small perturbations~\citep{Lewis-PartlySmooth,Lewis-PartlyTiltHessian}.  
A important and distinctive feature of our result is that, while the regularized $J$ is assumed to be partly smooth, the function $f$ is not partly smooth at $(\x,\th)=(\x_0,\th_0)$ relative to the manifold $\Mm_{\x_0} \times \Theta$ because of the indicator function (constrained problem) appearing in $\ener(\cdot,\th_0)$ when $\la=0$. Thus one cannot apply Theorem 5.7 of~\cite{Lewis-PartlySmooth}. We refer to Section~\ref{sec-sensitivity-lagrangian} for a discussion about this point. 
}



\section{Proofs}
\label{sec-proof-thm}

\LongShort{

\subsection{Sensitivity of the Lagrangian Problem}
\label{sec-sensitivity-lagrangian}

Before diving into the proof of Theorem~\ref{thm-stability}, we first show how the theory of partly smooth functions introduced in~\cite{Lewis-PartlySmooth} can be directly applied to study the sensitivity of~\eqref{eq-regularization-lagrangian} when $\mu>0$, and why some further refinement is needed to study the critical case $\mu=0$.

\begin{thm}
	Let $\x_\th$ be a solution of~\eqref{eq-regularization-lagrangian}.
        We assume that $J$ is locally partly smooth at $\x_\th$ relative to a set $\Mm$.
        If
	\eql{\label{eq-cond-sensi-lagr}
		\ker(\Corr) \cap T_{\x_\th} = \{0\}
		\qandq
		\frac{ \u - \Corr \x_\th }{\mu} \in \ri( \partial J(\x_\th) ), 
	} 
	then for $\th'$ close enough from $\th$, the solution $\x_{\th'}$ of~\eqref{eq-regularization-lagrangian} is unique and satisfies
	\eq{
          \x_{\th'} \in \Mm.
	}
\end{thm}

\begin{proof}
	This is an straightforward application of~\cite[Theorem 5.7]{Lewis-PartlySmooth}. Indeed, by the smooth perturbation rule \cite[Corollary~4.7]{Lewis-PartlySmooth}, the function $\ener$ is partly smooth at $(\x_\th,\th)$ relative to the manifold $\Mm \times \Theta$, and condition~\eqref{eq-cond-sensi-lagr} is exactly equivalent to $\x_\th$ being a strong minimizer of $\ener(\cdot,\th)$, see \cite[Definition~5.6]{Lewis-PartlySmooth}.   
\end{proof}

Condition~\eqref{eq-cond-sensi-lagr} is not very useful because it depends on the solution $\x_\th$ and not on the data to recover $\x_0$. The rationale behind Theorem~\ref{thm-stability} is to make $\th$ tends to $0$, and under the hypotheses of Theorem~\ref{thm-stability}, to obtain 
\eq{
	\x_\th \rightarrow \x_0
	\qandq
	\frac{ \u - \Corr \x_\th }{\mu} \rightarrow \etaL{\CorrL}.
} 
This is precisely what we need to prove to make the statement of the theorem correct. 



}

\newcommand{\xk}{\x_k}
\newcommand{\xt}{\tilde\x}
\newcommand{\Tk}{{T_k}}
\newcommand{\Corrk}{\Corr_k}

\subsection{Proof of Theorem~\ref{thm-stability}}
\label{sec-proof-thm-stability}

In order to prove Theorem~\ref{thm-stability}, we consider any sequence $\th_k = (\mu_k,u_k=\Corrk x_0 + \epsilon_k,\Corrk)_{k}$ where $\Phi_k \in \RR^{\P_k \times \N}$.
Assume that 
\eql{\label{eq-const-conv-proof}
	\pa{
		\Corrk, 
      	\epsilon_k \, \mu_k^{-1}, 
      	\mu_k
  	} 
	\longrightarrow 
	(\CorrL,0,0)~.
}
Then proving Theorem~\ref{thm-stability} boils down to showing that for $k$ large enough, the solution $\xk$ of $(\Pp_{\th_k})$  is unique and satisfies $\xk \in \Mm$.

\paragraph{Constrained problem.}

We consider the following non-smooth,in general non-convex, constrained minimization problem
\eql{\label{eq-noncvx}
	\xk \in  
	\uArgmin{\x \in \Mm \cap \Kk} \ener(\x,\th_k)
}
where $\Kk$ is an arbitrary fixed convex compact neighbourhood of $\x_0$. 


The following lemma first show the convergence of $\x_k$.

\begin{lem}\label{prop-conv-primal}
	Under condition~\eqref{eq-const-conv-proof}, 
	$\x_k \rightarrow \x_0$.
\end{lem}
\begin{proof}
	We denote $\norm{u}_{\Corr}^2 = \dotp{\Corr u}{u}$ for any non-negative definite matrix $\Gamma$. 
	We first note that~\eqref{eq-hyp-etaF} implies that $\x_0$ is the unique solution of~$(\Pp_{0,\CorrL \x_0,\CorrL})$. 
	By optimality of $\x_{k}$ one has $E(\x_k,\th_k) \leq E(\x_0,\th_k)$ and hence
	\begin{align*}
		\frac{1}{2}\norm{\x_k}_{\Corr_k} - \dotp{\x_k}{\Corr_k \x_0+\epsilon_k} + \mu_k J(\x_k) 
		\leq \frac{1}{2}\norm{\x_0}_{\Corr_k} - \dotp{\x_0}{\Corr_k \x_0+\epsilon_k} + \mu_k J(\x_0)	
	\end{align*} 
	which is equivalently stated as
	\eql{\label{eq-proof-Gamma-cv-1}
		\frac{1}{2}\norm{\x_k-\x_0}_{\Corr_k}^2 - \dotp{\x_k-\x_0}{\epsilon_k} + \mu_k J(\x_k) 
		\leq
		\mu_k J(\x_0).		
	}
	Since $\x_k \in \Kk$, the sequence $(x_k)_k$ is bounded, and we let $\x^\star$ be any accumulation point. 
	Taking the limit $k \rightarrow +\infty$ in~\eqref{eq-proof-Gamma-cv-1} and using~\eqref{eq-const-conv-proof} and continuity of the inner product shows that $\CorrL \x^\star = \CorrL \x_0$.
	Furthermore, since $\frac{1}{2}\norm{\x_k-\x_0}_{\Corr_k}^2 \geq 0$, \eqref{eq-proof-Gamma-cv-1} yields
	$	- \dotp{\x_k-\x_0}{\frac{\epsilon_k}{\mu_k}} +  J(\x_k) 
		\leq
		 J(\x_0).$		
	Taking the limit $k \rightarrow +\infty$ shows that $J(\x^\star) \leq J(\x_0)$. 
	Combining this with the previous claim that $\x^\star$ is a feasible point of $(\Pp_{0,\CorrL x_0,\CorrL})$ allows to conclude that $\x^\star$ is a solution of~$(\Pp_{0,\CorrL x_0,\CorrL})$. Since $\x_0$ is unique, this leads to $\x^\star = \x_0$. 
\end{proof}

We now aim at showing that for $k$ large enough, $\xk$ is the unique solution of $(\Pp_{\th_k})$. In order to do so, we make use of the following classical result, whose proof can be found for instance in~\cite{vaiter-analysis}. 

\begin{prop}\label{prop:uniqueness}
  	Let $\x \in \RR^\N$ such that $\frac{u - \Corr \x}{\mu} \in \ri( \partial J(\x) )$ and $\ker(\Corr) \cap T_{\x} = \{0\}$.
  	Then $\x$ is the unique solution of~\eqref{eq-regularization-lagrangian}.
\end{prop}

\paragraph{Convergence of the tangent model subspace.}

By definition of the constrained problem~\eqref{eq-noncvx}, $\xk \in \Mm$.
Moreover, since $\ener(\cdot,\theta_k)$ is partly smooth at $\x_0$ relative to $\Mm$, the sharpness property Definition~\ref{dfn-partly-smooth}\eqref{PS-Sharp} holds at all nearby points in the manifold $\Mm$ \cite[Proposition~2.10]{Lewis-PartlySmooth}. Thus as soon as $k$ is large enough, we have $\Tk = \tgtManif{\xk}{\Mm}$. Using the fact that $\Mm$ is of class $\Cdeux$, we get
\eql{\label{eq-conv-T}
	\Tk = \tgtManif{\xk}{\Mm}
	\longrightarrow
	\tgtManif{\x_0}{\Mm} = T
}
when~\eqref{eq-const-conv-proof} holds, 
where the convergence should be understood over the Grassmannian of linear subspaces with the same dimension (or equivalently, as the convergence of the projection operators $P_{\Tk} \rightarrow P_{T}$). Since $\ker(\CorrL) \cap T = \{0\}$, \eqref{eq-conv-T} implies that for $k$ large enough, when~\eqref{eq-const-conv-proof} holds, 
\eql{\label{eq-condition-inj-proof-1}
	\ker(\Corrk) \cap \Tk = \{0\},
} 
which we assume from now on.

\paragraph{First order condition.}

Let's take $\Kk=\mathbb{B}_r(\x_0)$ for $r$ sufficiently large. For any $\delta > 0$, $\exists K_\delta > 0$ such that $\forall k > K_\delta$, $\xk \in \mathbb{B}_\delta(\x_0)$. Thus, for $k$ large enough, i.e. $\delta$ sufficiently small, we indeed have $\xk \in \Int(\Kk)$.
Furthermore, it is easy to see that $\iota_{\Kk}$ is locally partly smooth at $\x_0$ relative to $\Kk$. Since is $J$ is also locally partly smooth at $\x_0$ relative to $\Mm$, the sum rule \cite[Corollary~4.6]{Lewis-PartlySmooth} shows that, for all sufficiently large $k$, when~\eqref{eq-const-conv-proof} holds and $\xk \in \Int(\Kk)$, $J+\iota_\Kk$ is locally partly smooth at $\xk$ relative to $\Mm \cap \Kk$, and then so is $\ener(\cdot,\theta_k) + \iota_{\Kk}$ by the smooth perturbation rule \cite[Corollary~4.7]{Lewis-PartlySmooth}. Therefore, \cite[Proposition 2.4(a)-(b)]{Lewis-PartlySmooth} applies, and it follows that $\xk$ is a critical point of \eqref{eq-noncvx} if, and only if,
\begin{align*}
    0 	\in \aff(\partial \ener(\xk,\theta_k) + N_{\Kk}(\xk))
    = \frac{\Corrk \xk - u_k}{\mu_k}  + \aff(\partial J(\xk)) 
    = \frac{\Corrk \xk - u_k}{\mu_k}  + e_{\xk} + \Tk^\perp .
\end{align*}
The first equality comes from the fact that $\ener(\cdot,\theta)$ is a closed convex function, and that the normal cone of $\Kk$ at $\xk$ vanishes on the interior points of $\Kk$, and the second one from the decomposability of the subdifferential.
Projecting this relation onto $\Tk$, we get, since $e_{\xk} \in \Tk$, 
\eql{\label{eq:firstordernoncvx}
	P_{\Tk} ( \Corrk \xk - u_k )  + \mu_k e_{\xk} = 0.
}

\paragraph{Convergence of the primal variables.}

Since both $\xk$ and $\x_0$ belong to $\Mm$, and partial smoothness implies that $\Mm$ is a manifold of class $\Cdeux$ around each of them, we deduce that each point in their respective neighbourhoods has a unique projection on $\Mm$~\cite{PoliquinRockafellar2000}.
In particular, $\xk=P_{\Mm}(\xk)$ and $\x_0=P_{\Mm}(\x_0)$. Moreover, $P_{\Mm}$ is of class $C^1$ near $\xk$ \cite[Lemma~4]{LewisMalick08}. Thus, $C^2$ differentiability shows that
\eq{
	\xk - \x_0 = P_{\Mm}(\xk) - P_{\Mm}(\x_0) = \mathrm{D} P_{\Mm}(\xk)(\xk - \x_0) + R(\xk)
}
where $R(\xk) = O( \norm{\xk-\x_0}^2 )$ and 
where $\mathrm{D} P_{\Mm}(\xk)$ is the derivative of $P_{\Mm}$ at $\xk$. Using \cite[Lemma~4]{LewisMalick08}, and recalling that $\Tk=\tgtManif{\xk}{\Mm}$ by the sharpness property, the derivative $\mathrm{D} P_{\Mm}(\xk)$ is given by 
$\mathrm{D} P_{\Mm}(\xk) = P_{\Tk}$.
Inserting this in  \eqref{eq:firstordernoncvx}, we get
\eq{
	P_{\Tk} \Corrk \pa{ P_{\Tk}(\xk - \x_0) + R(\xk) } - P_{\Tk} \epsilon_k  + \mu_k e_{\xk} = 0.
}
Using~\eqref{eq-condition-inj-proof-1}, $\Corr_{k,\Tk}$ has full rank, and thus
\eql{\label{eq:implicitprimal}
	\xk - \x_0 = \Corr_{k,\Tk}^+ \pa{
		\epsilon_k - \mu_k e_{\xk} - \Corrk R(\xk)
	},
}
where we also used that $\Tk^\perp \subset \ker(\Corr_{k,\Tk}^+)$. One has $\Corr_{k,\Tk}^{+} \rightarrow \CorrL^{+}$ so that $\Corr_{k,\Tk}^+ \Corrk = O(1)$ and  $\Corr_{k,\Tk}^{+}  = O(1)$.
Altogether, we thus obtain the bound
\eql{\label{eq-conv-txt-x0}
	\norm{ \xk - \x_0 } =  O\pa{ \norm{\epsilon_k}, \mu_k }.
}

\paragraph{Convergence of the dual variables.}

\newcommand{\etak}{\eta_k} 
\newcommand{\pk}{\p_k}

We define $\etak = \frac{u_k - \Corrk \xk}{\mu_k}$.
Arguing as above, and using \eqref{eq:implicitprimal} we have
\begin{align*}
	\mu_k \etak 	&= \epsilon_k + \Corrk (\x_0-\xk) 
			= \epsilon_k - \Corrk \Corr_{k,\Tk}^+ 
				\pa{ \epsilon_k - \mu_k e_{\xk} - \Corrk R(\xk) } \\
			&= \epsilon_k - \Corrk P_\Tk \Corr_{k,\Tk}^+ 
				\pa{ \epsilon_k - \mu_k e_{\xk} - \Corrk R(\xk) } \\
			&= P_{V_\Tk^\bot} \epsilon_k + P_{V_\Tk} \Corrk R(\xk) + \mu_k \Corrk \Corr_{k,\Tk}^+ e_{\xk}, 
\end{align*}
where we denoted $V_{\Tk} = \Im( \Corrk P_{\Tk} )$, and used that $\Im(\Corr_{k,\Tk}^+) \subset \Tk$.
We thus arrive at
\eq{
      \norm{ \etak - \etaL{\CorrL} } = 
      O\pa{
	      {\norm{\epsilon_k}}{\mu_k^{-1}},
	      \norm{ \Corrk \Corr_{k,\Tk}^+ e_{\xk} - \etaL{\CorrL} }, 
	      \norm{\Corrk} {\norm{\xk-\x_0}^2}{\mu_k^{-1}}
      }.      
}

Since $\Mm$ is a $C^2$ manifold, and by partial smoothness ($J$ is $C^2$ on $\Mm$), we have $\x \mapsto e_\x$ is $C^1$ on $\Mm$, one has
\eql{\label{eq-proof-tmp-1}
	\norm{ e_{\xk} - e } = O( \norm{\xk - \x_0} ).
}

Using the triangle inequality, we get
\eq{
	\norm{ \Corrk \Corr_{k,\Tk}^+  - \CorrL \CorrL_T^{+} } 
	\leq
	\norm{ \Corr_{k,\Tk}^+ } \norm{ \Corr_{k} - \CorrL }
	+ 
	\norm{ \CorrL } \norm{ \Corr_{k,\Tk}^+ - \CorrL_T^{+} }.
}
Again, since $\Corr_{k,\Tk}^+ \rightarrow \CorrL_{T}^+$, we have $\norm{ \Corr_{k,\Tk}^+ } = O(1)$. 
Moreover, $A \mapsto A^{+}$ is smooth at $A = \Corr_{T}$ along the manifold of matrices of constant rank, and $\Mm$ is a $C^2$ manifold near $\x_0$. Thus
\eq{
	\norm{ \Corr_{k,\Tk}^+ - \CorrL_T^{+} } = O( \norm{ \Corr_{k,\Tk} - \CorrL_T }  )
	= O( \norm{ \Corr_{k} - \CorrL },  \norm{P_{\Tk} - P_T} )
	= O( \norm{ \Corr_{k} - \CorrL }, \norm{\xk-\x_0}).
}
This shows that
\eql{\label{eq-proof-tmp-2}	
	\norm{ \Corrk \Corr_{k,\Tk}^+  - \CorrL \CorrL_T^{+} }=
	O( \norm{\Corrk-\CorrL}, \norm{\xk - \x_0} ).
}
Putting~\eqref{eq-proof-tmp-1} and \eqref{eq-proof-tmp-2} together implies
	$\norm{ \Corrk \Corr_{k,\Tk}^+ e_{\xk} - \etaL{\CorrL} }
	O( \norm{\Corrk-\CorrL}, \norm{\xk - \x_0} )$.
Altogether, we get the bound
\eq{\label{eq-conv-teta-Phi}
	\norm{\etak - \etaL{\CorrL}} =    
	O\pa{
	      {\norm{\epsilon_k}}{\mu_k^{-1}},
	      \norm{\xk-\x_0}, \norm{\Corrk-\CorrL},  
	      \norm{\Corrk} {\norm{\xk-\x_0}^2}{\mu_k^{-1}}
      }.
}
Since $\norm{\xk-\x_0}$ is bounded according to~\eqref{eq-conv-txt-x0}, we arrive at
\eql{\label{eq-conv-teta}
	\norm{\etak - \etaL{\CorrL}} =    
	O\pa{
		\norm{\Corrk-\CorrL}, 
	      { \norm{\epsilon_k} }{ \mu_k^{-1} }, 
	      \mu_k
      }.
}

\paragraph{Convergence inside the relative interior.}

Using the hypothesis that $\etaL{\CorrL} \in \ri(\partial J(\x_0))$, we will show that for $k$ large enough,  
\eql{\label{eq-condition-ri-proof-1}
	\etak \in \ri(\partial J(\xk)).
} 
Let us suppose this does not hold. Then there exists a sub-sequence of $\etak$, that we do not relabel for the sake of readability of the proof, such that 
\eql{\label{eq-etan-rbound}
	\etak \in \text{rbound}(\partial J(\xk)) ~.
}
According to~\eqref{eq-conv-teta} and Lemma~\ref{prop-conv-primal}, under \eqref{eq-const-conv-proof},
	$(\xk,\etak) \rightarrow (\x_0,\etaL{\CorrL})$.
Condition~\eqref{eq-etan-rbound} is equivalently stated as, for each $k$
\eql{\label{eq-contradict-rbound}
	\exists z_k \in T_{\xk}^\bot, \quad
	\foralls \eta \in \partial J(\xk), \quad
	\dotp{z_k}{\eta-\etak} \geq 0, 
}
where one can impose the normalization $\norm{z_k}=1$ by positive-homogeneity. Up to a sub-sequence (that for simplicity we still denote $z_k$ with a slight abuse of notation), since $z_k$ is in a compact set, we can assume $z_k$ approaches a non-zero cluster point $z^\star$.

Since $T_{\xk}^\bot \rightarrow T^\bot$ because $\Mm$ is a $C^2$ manifold, one has that $z^\star \in T^\bot$. We now show that 
\eql{\label{eq-proof-contradiction}
	\foralls v \in \partial J(\x_0), \quad
	\dotp{z^\star}{\eta-\etaL{\CorrL}} \geq 0.
}
Indeed, let us consider any $v \in \partial J(\x_0)$. In view of the continuity property in Definition~\ref{dfn-partly-smooth}\eqref{PS-DiffCont}
$\partial J$ is continuous at $\x_0$ along $\Mm$, so that since $\x_k \rightarrow \x_0$ there exists $v_k \in \partial J(\xk)$ with $v_k \rightarrow v$. Applying~\eqref{eq-contradict-rbound} with $\eta=v_k$ gives $\dotp{z_k}{v_k-\etak} \geq 0$. 
Taking the limit $k \rightarrow +\infty$ in this inequality leads to~\eqref{eq-proof-contradiction}, which contradicts the fact that $\etaL{\CorrL} \in \ri(\partial J(\x_0))$. 
In view of \eqref{eq-condition-ri-proof-1} and~\eqref{eq-condition-inj-proof-1}, using Proposition~\ref{prop:uniqueness} shows that $\xk$ is the unique solution of~\eqref{eq-regularization-lagrangian}. \qed

\subsection{Proof of Theorem~\ref{thm-consistency}}
\label{sec-proof-thm-consistency}

It is sufficient to check that \eqref{eq-condition-closeness} is in force with probability 1 as $\P \to +\infty$. Owing to classical results on convergence of sample covariances, which apply thanks to the assumption that the fourth order moments are finite, we get
$\Corr_\P - \CorrL = O_P\pa{\P^{-1/2}}$ and $\frac{1}{\P}\dotp{\xi_i}{w} = O_P\pa{\P^{-1/2}}$, 
where we used the assumption that $\EE\pa{\dotp{\xi_i}{w}} = 0$. As $p$ is fixed, it follows that
$\norm{\Corr_\P - \CorrL} = O_P\pa{\P^{-1/2}}$ and $\norm{\epsilon_\P} = O_P\pa{\P^{-1/2}}$.
Thus under the scaling~\eqref{eq-lambda-scaling}, we get
\begin{align*}
\pa{\norm{\Corr_\P - \CorrL}, { \norm{\epsilon_\P} }{\mu_\P^{-1}}, \mu_\P}
	&= \pa{O_P(\P^{-1/2}), \frac{1}{\mu_\P \P^{1/2}}O_P(1), o(1)} \\
	&= \pa{O_P(\P^{-1/2}), o(1)O_P(1), o(1)} 
	= \pa{O_P(\P^{-1/2}), o(1), o(1)} ~,
\end{align*}
which indeed converges to $0$ in probability. This concludes the proof. \qed

\subsection{Proof of Proposition~\ref{prop-instability}}
\label{sec-proof-prop-instability}

Let $\xk$ be a solution of $(\Pp_{\th_k})$.
Suppose that $\xk \in \Mm$.
In particular, $\xk$ is a solution of the non-convex minimization~\eqref{eq-noncvx}. Arguing as in the proof of Theorem~\ref{thm-stability}, we get the bound~\eqref{eq-conv-teta}, i.e.
\begin{equation}\label{eq-conv-teta-bis}
  	\norm{\eta_k - \etaL{\CorrL}} = O(\norm{\Corr_k-\CorrL}, \norm{\epsilon_k}/\mu_k, \mu_k)
  	\qwhereq
  	\eta_k = \frac{\u_k - \Corr_k \xk}{\mu_k} .
\end{equation}
In particular, $\norm{\eta_k - \etaL{\CorrL}} \to 0$.
Defining $K = d(\etaL{\CorrL},\partial J(\x))$, one has $K>0$ since $\etaL{\CorrL} \not\in \ri \partial J(\x_0)$.
Choosing $k$ large enough, the convergence of $\eta_k$ to $\etaL{\CorrL}$ implies that
\begin{equation}\label{eq-conv-sup-qt}
  d(\eta_k,\partial J(\x_0)) > K/2
\end{equation}
where $2$ can be changed to any arbitrary value.
Using the outer semi-continuity of the subdifferential, we get that
\begin{equation*}
  \forall \epsilon, \exists k_0, \forall k \geq k_0, \quad
  \partial J(x_k) \subseteq \partial J(\x_0) + B(0,\epsilon) .
\end{equation*}
In particular, $\eta_k \in \partial J(\x_0) + B(0,\epsilon)$ which implies that $d(\eta_k,\partial J(\x_0)) \leq \epsilon$, which is a contradiction to~\eqref{eq-conv-sup-qt}.
Hence, $\x_k \not\in \Mm$.


\section{Conclusion}

In this paper, we provided a unified analysis of the recovery performance when partly smooth functions are used to regularize linear inverse problems. This class of functions encompass all popular regularizers used in the literature. A distinctive feature of our work is that we provided for the first time a unified analysis together with a generalized ``irrepresentable condition'' to guarantee stable and correct identification of the low-complexity manifold underlying the original object. 
%
\OnlyLong{Our work also shows that model consistency is not only of theoretical interest, but also has practical implications because it can be observed after a finite number of iterations of a proximal splitting scheme (here the Forward-Backward). This could also be useful to speedup existing optimization methods by switching to a higher order optimization scheme exploiting the smoothness of the objective function along the smooth model manifold. }


\ifnipsfinal
\section*{Acknowledgements} 

The authors would like to thank Vincent Duval and J\'er\^ome Malick for fruitful discussions. This work has been supported by the European Research Council (ERC project SIGMA-Vision).
\fi

\small{
\bibliographystyle{plain}
\bibliography{bibliography}
}

\end{document}